\documentclass[11pt,reqno]{amsart}

\usepackage{lipsum}

\usepackage{graphicx}
\usepackage{array}
\usepackage{amssymb}
\usepackage{hyperref}
\usepackage{amsthm}
\usepackage{amsfonts}
\usepackage{amsmath}
\usepackage{bm}
\usepackage{mathrsfs}
\usepackage[all]{xy}
\usepackage{color}
\usepackage{subfigure}
\usepackage{tikz, tikz-cd}
\usepackage{enumerate}
\usepackage{anysize}
\usepackage{amscd}
\usepackage[letterpaper]{geometry}
\usepackage{multirow}
\usepackage{array}
\usepackage{booktabs}

\newtheorem{theorem}{Theorem}[section]
\newtheorem{proposition}[theorem]{Proposition}

\newtheorem{conjecture}[theorem]{Conjecture}

\theoremstyle{remark}
\newtheorem{question}{Question}

\theoremstyle{definition}
\newtheorem{definition}[theorem]{Definition}
\newtheorem{remark}[theorem]{Remark}

\numberwithin{equation}{section}
\numberwithin{figure}{section}
\numberwithin{table}{section}

\allowdisplaybreaks[3]

\begin{document}

 \title{Lecture notes on generalized Monge-Amp\`ere equations and subvarieties}

 \author{Gao Chen} \thanks{}
  \address{Institute of Geometry and Physics, University of Science and Technology of China, Shanghai, China, 201315}
 \email{chengao1@ustc.edu.cn}

\begin{abstract}
These are the lecture notes for the Morningside Center of Mathematics Geometry Summer School on August 15-20, 2022. These lectures sketch the results by Yau, Demailly-Paun, the author, and Datar-Pingali about generalized Monge-Amp\`ere equations and subvarieties and aim to use these results to study the Hodge conjecture.
\end{abstract}

\maketitle

\setcounter{tocdepth}{1}
\tableofcontents

\section{Introduction}

We begin with Calabi's conjecture which was solved by Yau \cite{Yau}.
\begin{theorem}[\cite{Yau}]
Let $(M,J,\omega_0,g_0)$ be a K\"ahler manifold and $f\in C^{\infty}(M)$ be a function such that
\begin{equation}
\int_M e^f \omega_0^n = \int_M \omega_0^n.
\end{equation}
Then there exists $\varphi\in C^{\infty}(M)$ such that
\begin{equation}
\omega_\varphi = \omega_0 + i\partial\bar\partial \varphi>0,
\end{equation}
and
\begin{equation}
\omega_\varphi^n = e^f \omega_0^n.
\end{equation}
\end{theorem}
\begin{remark}
Conversely, if there exists $\varphi$ such that
\begin{equation}
\omega_\varphi^n = e^f \omega_0^n,
\end{equation}
then
\begin{equation}
\int_M e^f \omega_0^n = \int_M \omega_\varphi^n = \int_M \omega_0^n.
\end{equation}
\end{remark}
\begin{remark}
When $c_1(M)=0$, we can choose $f$ such that $\omega_\varphi$ induces a K\"ahler-Ricci-flat metric. In honor of Calabi and Yau, people call K\"ahler-Ricci-flat metrics as Calabi-Yau metrics.
\end{remark}

This short course aims to generalize this result and use the generalization to study subvarieties.

If $M$ is K\"ahler, then a subset $V\subset M$ is called an analytic subvariety if, for any $x\in V$, there exists a neighborhood $U$ such that $x\in U\subset M$ and 
\begin{equation}
V\cap U = \{f_1=f_2=...=f_N=0\}
\end{equation}
for holomorphic functions $f_i$.
\begin{theorem}[\cite{Chow}]
If $M$ is projective, then any analytic subvariety is algebraic. In other words, $U$ can be chosen as a Zariski open set, and $f_i$ are polynomials.
\end{theorem}
\begin{theorem}[Sard]
Singular points on any analytic subvariety $V$ has measure zero, where $p$ is smooth if $T_p V=\mathbb{R}^{dim V}$ and singular if $p$ is not smooth.
\end{theorem}

\begin{theorem}[\cite{Lelong}]
Any analytic subvariety $V$ defines a closed positive current.
\end{theorem}
In the above, a current $\omega$ means that $\int_M \omega\wedge\chi$ formally make sense for all smooth $(p,p)$-forms $\chi$. In particular, any $(n-p,n-p)$-forms determines a current. The integral on smooth points of $V$ also determines a current $[V]$ by
\begin{equation}
\int_M [V]\wedge \chi =\int_{V^{\mathrm{smooth}}}\chi.
\end{equation}
If $\omega$ is a form, then
\begin{equation}
\int_M d\omega\wedge\chi = \int_M d(\omega\wedge\chi) - \int_M\omega\wedge d\chi = - \int_M\omega\wedge d\chi.
\end{equation}
This formula also defines $d$ acting on currents. A current $\omega$ is called a closed current if $d\omega=0$.
A $(p,p)$-form $\omega$ is positive if 
\begin{equation}
\omega\wedge\beta_1\wedge\bar\beta_1...\wedge\beta_{n-p}\wedge\bar\beta_{n-p}>0
\end{equation}
for any $\beta_i\in\Omega^1(M)$. $\omega$ is called strongly positive if
$\omega=\sum_i \beta_{i,1}\wedge\bar\beta_{i,1}...\wedge\beta_{i,p}\wedge\bar\beta_{i,p}$ for $\beta_{i,j}\in\Omega^1(M)$. It can be proved that for (1,1)-forms, strongly positive is the same as positive. We call a current $\omega$ as a positive (strongly positive) current if for any strongly positive (positive) forms, $\int_M \omega\wedge \chi>0$. In particular, if $\omega$ is a form, then $\omega$ is a positive (strongly positive) current if and only if $\omega$ is a positive (strongly positive) form. See \cite{Demailly} for more details on positive currents and subvarieties.

Lelong's theorem implies that any analytic subvariety $V$ defines a closed positive current $[V]$. Actually, it can be proved that $[V]$ determines an element in $H^{n-p,n-p}(M,\mathbb{C})\cap H^{2n-2p}(M,\mathbb{Z})$. One of the seven Clay problems is
\begin{conjecture}[\cite{Hodge}]
If $M$ is a smooth projective manifold, then $H^{n-p,n-p}(M,\mathbb{C})\cap H^{2n-2p}(M,\mathbb{Z})$ is the $\mathbb{Z}$-span of subvarieties.
\end{conjecture}
Atiyah and Hirzebruch \cite{AtiyahHirzebruch} proved that Hodge's original conjecture does not hold. The modified version is the following
\begin{conjecture}[Hodge, modified version]
If $M$ is a smooth projective manifold, then $H^{n-p,n-p}(M,\mathbb{C})\cap H^{2n-2p}(M,\mathbb{Q})$ is the $\mathbb{Q}$-span of subvarieties.
\end{conjecture}

The goal of this minicourse is to try to generalize the Monge-Amp\`ere equation to study the Hodge conjecture. The first remark is to compare projective manifolds and K\"ahler manifolds. There is a weaker version on K\"ahler manifolds:
\begin{conjecture}
If $M$ is a K\"ahler manifold, then $H^{n-p,n-p}(M,\mathbb{C})\cap H^{2n-2p}(M,\mathbb{Q})$ is the $\mathbb{Q}$-span of Chern classes of coherent sheaves.
\end{conjecture}
Voisin \cite{Voisin} proved that this weaker version is not true. Therefore, the modified version of the Hodge conjecture is also false on K\"ahler manifolds. Thus, to prove the Hodge conjecture, it is very important to use the projective condition in an essential way.

The projective condition means that there exists a K\"ahler form $\omega_0$ such that $[\omega_0]=c_1(L)$ for an ample line bundle $L$. So we expect this (1,1) form to play an important role. Thus, there might be five steps to solve the Hodge conjecture.
\begin{enumerate}
\item Relate a $(1,1)$ class with the real span of subvarieties.
\item Relate two $(1,1)$ classes with the real span of subvarieties.
\item Relate an $(1,1)$ class, a $(p,p)$ class with the real span of subvarieties.
\item Relate a $(p,p)$ class with the real span of subvarieties.
\item Reduce the real span of subvarieties to the rational span of subvarieties.
\end{enumerate}

For the first step, using Calabi's conjecture solved by Yau \cite{Yau}, Demailly and Paun \cite{DemaillyPaun} proved the following:
\begin{theorem}[\cite{DemaillyPaun}]
The space of all K\"ahler classes is one of the connected components of $P$ which consists of (1,1)-forms $[\omega]$ such that $\int_{V^p}\omega^p>0$ for all $p$-dimensional analytic subvariety $V^p$.
\end{theorem}
For the second step, the author proved the following:
\begin{theorem}[\cite{Chen}]
Let $M^n$ be a K\"ahler manifolds with K\"ahler forms $\chi$ and $\omega_0$. Let $c$ be the constant such that
\begin{equation}
\int_M c_0\omega_0^n = n \chi \wedge \omega_0^{n-1}.
\end{equation}
Then the J-equation
\begin{equation}
c_0\omega_\varphi^n = n \chi \wedge \omega_\varphi^{n-1}
\end{equation}
is solvable if and only if there exists $\epsilon>0$ independent of $V$ such that
\begin{equation}
\int_{V^p} c_0\omega_0^p - p \chi\wedge\omega_0^{p-1}\ge \epsilon (n-p) \int_{V^p} \omega_0^p.
\end{equation}
for all analytic subvariety $V^p$.
\end{theorem}
Song \cite{Song} improved this result to allow $\epsilon$ to depend on $V$.
For the third step, the question is the following:
\begin{question}
Let $\omega_0$ be a K\"ahler form. Find out a  condition on $(n-k,n-k)$ forms $\chi_k$ such that
\begin{equation}
\omega_\varphi^n=\sum_{k=0}^{n-1}\chi_k\wedge\omega_\varphi^k
\end{equation}
is solvable if and only if
\begin{equation}
\int_{V^p} \omega_0^p-\sum_{k=n-p}^{n-1}\frac{k!}{n!}\frac{p!}{(k+p-n)!}\chi_k\wedge\omega_0^{k+p-n} \ge 0
\end{equation}
for all subvarieties $V^p\subset M^n$, with equation if and only if $V^p=M^n$.
\end{question}
Datar and Pingali \cite{DatarPingali} proved that this is true if $M$ is projective, and $\chi_k=c_k\chi^{n-k}$ for a K\"ahler form $\chi$ and non-negative but not all $0$ constants $c_k$. The projective condition plays an important role in this result, which is what we hope. Remark that for the deformed Hermitian-Yang-Mills equation, some of $\chi_k$ are negative. So the best condition should involve $\chi_k$ as a tuple rather than studying them individually.

For the fourth step, recall that Kleiman \cite{Kleiman} proved that if $M$ is projective, then $[\omega]$ is nef (i.e. $[\omega+t\chi]$ is a K\"ahler class for all $t>0$, where $\chi$ is another K\"ahler form.) if and only if $\int_C \omega \ge 0$ for all curves $C$. The idea is to take the intersection with ample hyperplanes to reduce the dimension of subvarieties. If one can solve the third step, then the next question is
\begin{question}
Let $M$ be a smooth projective manifold. Then what can we say about a $(p,p)$-form $\omega$ such that
\begin{equation}
\int_V \omega \ge 0
\end{equation}
for all $p$-dimensional subvariety $V$?
\end{question}

If we can solve the fourth step, then by duality, we would have a full understanding of the closure of the real span of $p$-dimensional subvarieties. Then we still need to study the convergence problem to reduce the closure of the real span to the rational span. We also need the cancellation of positive coefficients with negative coefficients. The counterexample by Babaee and Huh \cite{BabaeeHuh} shows that this can not be achieved directly by an arbitrary subvariety in a cohomology class. To solve this problem, we may need to find a better representative instead. This may be achieved by studying the locus where the corresponding flow diverges. Then we need to understand the bubble tree structure for the convergence of subvarieties with certain bounds in order to fully understand the cohomology class.

The goal of this minicourse is to sketch the results of Yau \cite{Yau}, Demailly-Paun \cite{DemaillyPaun}, the author \cite{Chen}, and Datar-Pingali \cite{DatarPingali}, so that students of this minicourse and readers of these lecture notes may make further progress with this approach.

\section{Analytic existence theorems}
\subsection{Introduction}
In this section, we sketch the proof of the analytic existence theorem. Recall that Yau \cite{Yau} proved that $\omega_0>0$ implies that
\begin{equation}
\omega_\varphi^n=e^f\omega_0^n
\end{equation}
can be solved as long as the integral is equal to each other.
For the J-equation, the following result was proved by the author \cite{Chen} as an improvement of Song-Weinkove's result \cite{SongWeinkove}.
\begin{theorem}
If
\begin{equation}
c\omega_0^{n-1}-(n-1)\chi\wedge\omega_0^{n-2}>0,
\end{equation}
\begin{equation}
f>-\frac{1}{2n}(\frac{1}{c})^{n-1},
\end{equation}
and
\begin{equation}
\int_M f\chi^n=c\int_M\omega_0^n-\int_M n\chi\wedge\omega_0^{n-1} \ge 0,
\end{equation}
then there exists $\omega_\varphi=\omega_0+i\partial\bar\partial\varphi>0$
such that
\begin{equation}
c\omega_\varphi^{n-1}-(n-1)\chi\wedge\omega_\varphi^{n-2}>0
\end{equation}
and
\begin{equation}
f\chi^n=c\omega_\varphi^n-n\chi\wedge\omega_\varphi^{n-1}.
\end{equation}
\end{theorem}
To understand the cone condition
\begin{equation}
c\omega_\varphi^{n-1}-(n-1)\chi\wedge\omega_\varphi^{n-2}>0,
\end{equation}
at each point $x\in M$, we write 
\begin{equation}
\chi=i \delta_{i\bar j} d z^i\wedge d\bar z^j,
\end{equation}
\begin{equation}
\omega_\varphi=i \lambda_i\delta_{i\bar j} d z^i\wedge d\bar z^j.
\end{equation}
Then the cone condition is equivalent to saying that
\begin{equation}
\sum_{i\not=k}\frac{1}{\lambda_i}<c
\end{equation}
for all $k=1,2,...,n$.
The J-equation can be written as
\begin{equation}
\sum_{i=1}^{n}\frac{1}{\lambda_i}=c,
\end{equation}
and the twisted version is
\begin{equation}
\sum_{i=1}^{n}\frac{1}{\lambda_i}+\frac{f}{\prod_{i=1}^{n}\lambda_i}=c.
\end{equation}
Remark that the condition on $f$ implies that if 
\begin{equation}
\sum_{i=1}^{n}\frac{1}{\lambda_i}+\frac{f}{\prod_{i=1}^{n}\lambda_i}=c,
\end{equation}
and
\begin{equation}
\sum_{i\not=k}\frac{1}{\lambda_i}\le c,
\end{equation}
then
\begin{equation}
\sum_{i\not=k}\frac{1}{\lambda_i}< c.
\end{equation}
In other words, if we use the continuity method, then the cone condition is preserved along the path. To prove this, if
\begin{equation}
\sum_{i\not=k}\frac{1}{\lambda_i}\le c,
\end{equation}
then $c\ge\frac{1}{\lambda_i}$. So $\frac{1}{\prod_{i\not=k}\lambda_i}\le c^{n-1}$.
So
\begin{equation}
\begin{split}
c = \sum_{i=1}^{n}\frac{1}{\lambda_i}+\frac{f}{\prod_{i=1}^{n}\lambda_i}\\
= \frac{1}{\lambda_k}+\sum_{i\not=k}\frac{1}{\lambda_i}+\frac{f}{\prod_{i\not=k}\lambda_i}\frac{1}{\lambda_k}\\
= \sum_{i\not=k}\frac{1}{\lambda_i}+\frac{1}{\lambda_k}(1+\frac{f}{\prod_{i\not=k}\lambda_i}).
\end{split}
\end{equation}
Since
\begin{equation}
1+\frac{f}{\prod_{i\not=k}\lambda_i} > 1 - \frac{1}{2n}(\frac{1}{c})^{n-1} c^{n-1} > \frac{1}{2},
\end{equation}
we see that the cone condition is satisfied.

In Datar-Pingali's paper \cite{DatarPingali}, they proved that there exists a constant $\epsilon>0$ depending only on dimensions and the classes $[\omega_0]$, $[\chi]$ such that if the cone condition
\begin{equation}
n \omega_0^{n-1} - \sum_{k=1}^{n-1}c_k k \chi^{n-k}\wedge \omega_0^{k-1}>0
\end{equation}
is satisfied, and $f>-\epsilon$ is a function such that
\begin{equation}
\int_M f\chi^n=\int_M \omega_0^n - \sum_{k=1}^{n-1}\int_M c_k \chi^{n-k}\wedge\omega_0^k\ge 0,
\end{equation} 
then there exists $\omega_\varphi=\omega_0+i\partial\bar\partial\varphi>0$ satisfying the condition and solves the twisted generalized Monge-Amp\`ere equation
\begin{equation}
f\chi^n = \omega_\varphi^n - \sum_{k=1}^{n-1} c_k \chi^{n-k}\wedge \omega_\varphi^k.
\end{equation}

All of the analytic existence theorems were proved using the continuity method. Yau \cite{Yau} used the path
\begin{equation}
\omega_\varphi^n = e^{tf} \omega_0^n,
\end{equation}
the author \cite{Chen} used the path
\begin{equation}
\chi_t=t\chi+(1-t)\frac{c}{n}\omega_0,
\end{equation}
and the path
\begin{equation}
f_t = t \frac{\int_M f\chi^n}{\int_M \chi^n} + (1-t) f,
\end{equation}
and Datar-Pingali \cite{DatarPingali} used the path
\begin{equation}
\omega_t^n = t (\sum_{k=1}^{n-1} c_k\chi^{n-k}\wedge\omega_t^k +f\chi^n) + (1-t) c_0 \chi^n.
\end{equation}
In all of these cases, they proved that the conditions in the analytic existence theorems such as $f>-\epsilon$ and the integral conditions are all satisfied, and the cone condition is preserved along the path. So it suffices to prove the openness and an a priori estimate in all of these cases with the cone condition.
To prove the openness, Yau's path \cite{Yau} is
\begin{equation}
\omega_\varphi^n = e^{tf}\omega_0^n.
\end{equation}
So
\begin{equation}
\log(\frac{\omega_\varphi^n}{\omega_0^n})=tf.
\end{equation}
So
\begin{equation}
f=tr_{\omega_{\varphi_t}} \frac{d}{dt} \omega_{\varphi_t} = \Delta_{\omega_{\varphi_t}}\frac{d}{dt}\varphi_t.
\end{equation}
The Laplacian equation can be solved with the integral condition. For the J-equation,
\begin{equation}
f_t\chi_t^n = c\omega_t^n - n \chi_t\wedge\omega_t^{n-1}.
\end{equation}
So
\begin{equation}
cn\omega_t^{n-1}\wedge \frac{d}{dt} \omega_t - n(n-1) \chi_t \wedge \omega_t^{n-2}\wedge \frac{d}{dt} \omega_t
\end{equation}
can be written in terms of terms without $\frac{d}{dt} \omega_t$. The cone condition implies that 
\begin{equation}
cn\omega_t^{n-1} - n(n-1) \chi_t \wedge \omega_t^{n-2}
\end{equation}
is positive. So it can be proved that this is also an elliptic equation. Thus, the openness is also insured by the integral condition. The cone condition is used in Datar-Pingali's work \cite{DatarPingali} in a similar way.

\subsection{$L^1$-estimate}

We use Yau's method \cite{Yau} to prove the $L^1$-bound in all of the cases above. As long as $\omega_0+i\partial\bar\partial\varphi>0$,
\begin{equation}
tr_{\omega_0} (\omega_0+i\partial\bar\partial\varphi) = n+\Delta_{\omega_0}\varphi>0.
\end{equation}
Assume that $\int_M \varphi\omega_0^n=0$, then there exists a Green function $G:M\times M\to \mathbb{R}$ such that
\begin{equation}
\varphi(p)=-\int_M G(p,q) \Delta \varphi(q) dq.
\end{equation}
There exists a constant $K$ such that $G\ge -K$, so
\begin{equation}
\varphi(p)=-\int_M (G(p,q)+K) \Delta \varphi(q) dq \le n\int_M (G(p,q)+K) dq \le C.
\end{equation}
In other words, if $\int_M \varphi\omega_0^n=0$, then $\sup_M \varphi\le C$. If we define
\begin{equation}
\tilde\varphi = \varphi-(\sup_M\varphi+1),
\end{equation}
then
\begin{equation}
\sup_M\tilde\varphi = -1,
\end{equation}
and
\begin{equation}
\int_M |\tilde\varphi| = |\int_M \tilde\varphi| \le C.
\end{equation}
We can replace $\varphi$ by $\tilde\varphi$ so that we can assume that
\begin{equation}
\sup_M\varphi=-1,
\end{equation}
and
\begin{equation}
\int_M |\varphi|\le C.
\end{equation}
This proves the $L^1$-bound in all the cases.

\subsection{$C^0$-$C^2$ bounds}

Then we need $C^0$-$C^2$ bounds. In Yau's paper \cite{Yau}, the equation is
\begin{equation}
f = \log (\frac{\omega_\varphi^n}{\omega_0^n}).
\end{equation}
So
\begin{equation}
Df = \omega^{\varphi,i\bar j} D\omega_{\varphi, i\bar j} + ...,
\end{equation}
and
\begin{equation}
D^2 f = \omega^{\varphi,i\bar j} D^2\omega_{\varphi, i\bar j} -\omega^{\varphi,i\bar l} (D\omega_{\varphi, k\bar l}) \omega^{\varphi,k\bar j} (D\omega_{\varphi, i\bar j}) + ...,
\end{equation}
The term $\omega^{\varphi,i\bar j} D^2\omega_{\varphi, i\bar j}$ looks like $\Delta_\varphi(\Delta \varphi)$.
A more detailed calculation in \cite{Yau} proves that
\begin{equation}
\Delta_\varphi(\Delta \varphi) \ge \Delta f + \varphi_{k\bar n\bar l}\varphi_{i\bar j l}\omega^{k \bar j}_{\varphi} \omega_{\varphi}^{i \bar n} - C (\sum_{i,l} \frac{1+\varphi_{i, \bar i}}{1+\varphi_{l,\bar l}}-m^2).
\end{equation}
The third order derivative term helps us but we need to cancel the last term. Note that
\begin{equation}
\Delta_{\omega_{\varphi}}\varphi = \sum_i \frac{\varphi_{i\bar i}}{1+\varphi_{i\bar i}}
\end{equation}
looks like the last term. With the help of the second term, Yau \cite{Yau} manages to prove that
\begin{equation}\Delta_{\varphi} (e^{-A\varphi}(n+\Delta\varphi)) \ge -e^{-A\varphi} \Big( -C-C(n+\Delta\varphi)+C(n+\Delta\varphi)(\sum_i\frac{1}{1+\varphi_{i\bar i}})\Big)
\end{equation}
for any large enough constant $A$ and a constant $C$ depending on $A$.  Let $\lambda_i=1+\varphi_{i\bar i}$, then
\begin{equation}
(\sum_i \frac{1}{\lambda_i})^{n-1} \ge C \sum_{i} \frac{1}{\prod_{k\not=i}\lambda_k} = C\frac{\sum_i \lambda_i} {\prod_k \lambda_k} \ge C(n+\Delta\varphi)
\end{equation}
because the equation is that $\prod_k \lambda_k=e^f$.
Thus
\begin{equation}\Delta_{\varphi} (e^{-A\varphi}(n+\Delta\varphi)) \ge e^{-A\varphi} \Big( -C-C(n+\Delta\varphi)+C(n+\Delta\varphi)^{\frac{n}{n-1}}\Big)
\end{equation}
At the maximal point of $e^{-A\varphi}(n+\Delta\varphi)$,
\begin{equation}(n+\Delta\varphi)^{\frac{n}{n-1}} \le C (n+\Delta\varphi) + C.
\end{equation}
This implies that $n+\Delta\varphi\le C$ at this point.
So
\begin{equation}e^{-A\varphi}(n+\Delta\varphi) \le C e^{-A \inf\varphi}
\end{equation}
at this point and therefore at all points on $M$. In other words,
\begin{equation}n+\Delta\varphi \le C e^{A(\varphi-\inf\varphi)}
\end{equation}
on $M$ for all large enough $A$ and a constant $C$ depending on $A$.
To get the $C^0$ bound, recall that
\begin{equation}\Delta_{\varphi} (e^{-A\varphi}(n+\Delta\varphi)) \ge e^{-A\varphi} \Big( C(n+\Delta\varphi)-C\Big)
\end{equation}
using the inequality that
\begin{equation}(n+\Delta\varphi)^{\frac{n}{n-1}} \ge \frac{1}{\epsilon} (n+\Delta\varphi)-C_\epsilon
\end{equation}
for all $\epsilon>0$.
Moreover,
\begin{equation}
\int_M e^f \Delta_{\varphi} (e^{-A\varphi}(n+\Delta\varphi)) \omega_0^n = \int_M \Delta_{\varphi} (e^{-A\varphi}(n+\Delta\varphi)) \omega_\varphi^n =0.
\end{equation}
So
\begin{equation}
0\ge C\int_M e^{-A\varphi}(n+\Delta\varphi) - C\int_M e^{-A\varphi}. 
\end{equation}
This implies that
\begin{equation}
\int_M e^{-N\varphi} \Delta\varphi\le C\int_M e^{-N\varphi}
\end{equation}
for all large enough constant $N$ and a constant $C$ depending on $N$. Note that
\begin{equation}
\Delta e^{-N\varphi} = N^2 e^{-N\varphi} |\nabla\varphi|^2 -N e^{-N\varphi}\Delta\varphi.
\end{equation}
So
\begin{equation}
0 = \int_M N^2 e^{-N\varphi} |\nabla\varphi|^2 - \int_M N e^{-N\varphi}\Delta\varphi = 4 \int_M |e^{-N\varphi/2}\frac{N}{2}\nabla\varphi|^2 - \int_M N e^{-N\varphi}\Delta\varphi.
\end{equation}
This implies that
\begin{equation}
\int_M |\nabla e^{-N \varphi/2}|^2 \le \frac{N}{4} \int_{M} e^{-N\varphi}\Delta\varphi \le C\int_M e^{-N\varphi}.
\end{equation}
If there exists a large enough constant $N$ and a sequence $\varphi_i$ depending on $N$ such that $\int_M e^{-N\varphi_i}\to\infty$. Let $\tilde\varphi_i=\varphi_i-c_i$ such that $\int_M e^{-N\tilde\varphi_i}=1$. Then $||e^{-N\tilde\varphi_i/2}||_{W^{1,2}}\le C$. This implies that after taking a subsequence, $e^{-N\tilde\varphi_i/2}\to F$ in $L^2$-norm to a function $F$ such that $||F||_{L^2}=1$. However, for any $\lambda>0$,
\begin{equation}Vol(e^{-N\tilde\varphi_i}>\lambda) = Vol(-\varphi_i > \frac{\log \lambda}{N}+\frac{1}{N}\log\int_M e^{-N\varphi_i}) \le \frac{\int_M |\varphi_i|}{\frac{\log \lambda}{N}+\frac{1}{N}\log\int_M e^{-N\varphi_i}} \to 0.
\end{equation}
So $F=0$, this is a contradiction. Thus, we have proved that
$\int_M e^{-N\varphi}\le C$ for any large enough constant $N$ and a constant $C$ depending on $N$.

Recall that
\begin{equation}
-n\le \Delta\varphi \le -n + C e^{A(\varphi-\inf \varphi)} \le Ce^{-A\inf\varphi}.
\end{equation}
By Schauder's estimate,
\begin{equation}
\sup_M |\nabla \varphi| \le C \int_M |\varphi| + C e^{-A\inf\varphi} \le C_1 e^{-A\inf\varphi}.
\end{equation}
So on a ball with center achieving $\inf\varphi$ and radius $C_1^{-1} e^{A\inf\varphi} (-\inf \varphi)/2$, we have
\begin{equation}
\inf \varphi \le \varphi \le \inf\varphi/2.
\end{equation}
So
\begin{equation}
C\ge \int_M e^{-N\varphi} \ge C e^{-N\inf\varphi/2} (C^{-1}e^{A\inf\varphi})^n (-\inf\varphi)^n.
\end{equation}
If we choose $N>>A$, then $|\inf\varphi|\le C$. This implies that
\begin{equation}0\le n+\Delta\varphi \le C.\end{equation}

Now we consider the J-equation as in \cite{SongWeinkove} and \cite{Chen}. Yau's proof takes second derivatives of the equation. In our case, the twisted version is that
\begin{equation}
\omega_\varphi^{i\bar j}\chi_{i\bar j}+f\frac{\det \chi_{\alpha\bar\beta}}{\det\omega_{\varphi,\alpha\bar\beta}}=c.
\end{equation}
So
\begin{equation}
-\omega_\varphi^{i\bar l}(\omega_{\varphi,m\bar l k})\omega_\varphi^{m\bar j}\chi_{i\bar j}-f\frac{\det \chi_{\alpha\bar\beta}}{\det\omega_{\varphi,\alpha\bar\beta}}\omega_{\varphi}^{i\bar j}\omega_{\varphi,i\bar j k}=...
\end{equation}
Thus, if we define
\begin{equation}
\tilde\Delta u = \omega_\varphi^{i\bar l}u_{m\bar l}\omega_\varphi^{m\bar j}\chi_{i\bar j} + f \frac{\det \chi_{\alpha\bar\beta}}{\det\omega_{\varphi,\alpha\bar\beta}}\omega_{\varphi}^{i\bar j}u_{i\bar j},
\end{equation}
then we can write $\tilde\Delta(\nabla \varphi)$ as lower order terms. If we take one more derivative, then a long calculation in \cite{Chen} based on \cite{SongWeinkove} shows that
\begin{equation}\tilde\Delta (\log tr_\chi \omega_{\varphi}) \ge -C.
\end{equation}
Recall that Yau \cite{Yau} used $e^{-A\varphi}(n+\Delta\varphi) = e^{-A\varphi}tr_{\omega_0}\omega_\varphi$. In our case, note that $C^{-1}\omega_0 \le \chi \le C\omega_0$ and we use
\begin{equation}
\log (e^{-A\varphi}tr_\chi\omega_\varphi) = \log tr_\chi \omega_{\varphi}-A\varphi
\end{equation}
instead. We have
\begin{equation}
\tilde\Delta (\log tr_\chi \omega_{\varphi}-A\varphi) \ge -C-A\tilde\Delta\varphi.
\end{equation}
Thus, for any $\epsilon>0$, if we choose $A$ large enough depending on $\epsilon$, then at the maximal point of $\log tr_\chi \omega_{\varphi}-A\varphi$, $\tilde\Delta\varphi\ge -C/A=-\epsilon$.
Note that
\begin{equation}
\tilde\Delta \varphi = \omega_\varphi^{i\bar l}(\omega_{\varphi,m\bar l}-\omega_{0,m\bar l})\omega_\varphi^{m\bar j}\chi_{i\bar j} + f \frac{\det \chi_{\alpha\bar\beta}}{\det\omega_{\varphi,\alpha\bar\beta}}\omega_{\varphi}^{i\bar j}(\omega_{\varphi,i\bar j}-\omega_{0,i\bar j}).
\end{equation}
At the maximal of $\log tr_\chi \omega_{\varphi}-A\varphi$, we diagonalize $\omega_{0,i\bar j}=\delta_{i\bar j}$ and $\omega_{\varphi,i\bar j}=\mu_i\delta_{i\bar j}$. Then
\begin{equation}-\epsilon \le \sum_i \frac{\chi_{i\bar i}}{\mu_i} - \sum_i \frac{\chi_{i\bar i}}{\mu_i^2} + f \frac{\det \chi_{\alpha\bar\beta}}{\det\omega_{\varphi,\alpha\bar\beta}}\sum_i(1-\frac{1}{\mu_i}).
\end{equation}
We want to complete the square for $\sum_i \frac{\chi_{i\bar i}}{\mu_i} - \sum_i \frac{\chi_{i\bar i}}{\mu_i^2}$, so we see that
\begin{equation}\begin{split}
-\epsilon \le - \sum_i \frac{\chi_{i\bar i}}{\mu_i^2} + 2\sum_i \frac{\chi_{i\bar i}}{\mu_i} - \sum_i \frac{\chi_{i\bar i}}{\mu_i} + f \frac{\det \chi_{\alpha\bar\beta}}{\det\omega_{\varphi,\alpha\bar\beta}}\sum_i(1-\frac{1}{\mu_i})\\
=\sum_{i\not=k} (-\frac{\chi_{i\bar i}}{\mu_i^2} + 2\sum_i \frac{\chi_{i\bar i}}{\mu_i}) - \frac{\chi_{k\bar k}}{\mu_k^2} + 2\frac{\chi_{k\bar k}}{\mu_k} - \sum_i \frac{\chi_{i\bar i}}{\mu_i} + f \frac{\det \chi_{\alpha\bar\beta}}{\det\omega_{\varphi,\alpha\bar\beta}}\sum_i(1-\frac{1}{\mu_i})\\
\le \sum_{i\not=k} \chi_{i\bar i}- \frac{\chi_{k\bar k}}{\mu_k^2} + 2\frac{\chi_{k\bar k}}{\mu_k} - c + f \frac{\det \chi_{\alpha\bar\beta}}{\det\omega_{\varphi,\alpha\bar\beta}}(1+\sum_i(1-\frac{1}{\mu_i})).
\end{split}
\end{equation}
We have used the equation that
\begin{equation}
c=\sum_i \frac{\chi_{i\bar i}}{\mu_i} + f \frac{\det \chi_{\alpha\bar\beta}}{\det\omega_{\varphi,\alpha\bar\beta}}.
\end{equation}
The cone condition for $\omega_0$ implies that $\sum_{i\not=k} \chi_{i\bar i} < c$. Since $M$ is compact, we can choose $\epsilon$ such that $\sum_{i\not=k} \chi_{i\bar i}\le c-2\epsilon$.
Then
\begin{equation}
\epsilon 
\le - \frac{\chi_{k\bar k}}{\mu_k^2} + 2\frac{\chi_{k\bar k}}{\mu_k} + f \frac{\det \chi_{\alpha\bar\beta}}{\det\omega_{\varphi,\alpha\bar\beta}}(1+\sum_i(1-\frac{1}{\mu_i})).
\end{equation}

If
\begin{equation}
f \frac{\det \chi_{\alpha\bar\beta}}{\det\omega_{\varphi,\alpha\bar\beta}}(1+\sum_i(1-\frac{1}{\mu_i})) \ge \epsilon/2,
\end{equation}
since the cone condition for $\omega_\varphi$ implies that $\sum_{i\not=k} \frac{\chi_{i\bar i}}{\mu_i} < c$, we see that $\frac{1}{\mu_i}\le C$ for all $i$. So 
\begin{equation}
|1+\sum_i(1-\frac{1}{\mu_i})|\le C.
\end{equation}
We also know that $|f|\le C$, so $\frac{\det \chi_{\alpha\bar\beta}}{\det\omega_{\varphi,\alpha\bar\beta}}\ge 1/C$ for another constant $C$. Since the eigenvalue for $\chi^{i\bar j}\omega_{\varphi,k\bar j }$ is bounded from below by the cone condition on $\omega_\varphi$ and the product is bounded, we see that all the eigenvalues are bounded from above. This implies a bound on $tr_\chi\omega_\varphi$. 

If \begin{equation}
f \frac{\det \chi_{\alpha\bar\beta}}{\det\omega_{\varphi,\alpha\bar\beta}}(1+\sum_i(1-\frac{1}{\mu_i})) \le \epsilon/2,
\end{equation}
then \begin{equation}\epsilon/2\le - \frac{\chi_{k\bar k}}{\mu_k^2} + 2\frac{\chi_{k\bar k}}{\mu_k}. \end{equation} This also implies a bound on $\mu_k$ and therefore a bound on $tr_\chi\omega_\varphi$ at this point.

In conclusion, we have proved that for all large enough $A$, at the maximal point of $\log tr_\chi \omega_{\varphi}-A\varphi$, $tr_\chi \omega_{\varphi}\le C$. So $\log tr_\chi \omega_{\varphi}-A\varphi\le C-A\inf \varphi$ at this point and therefore at all points on $M$. Thus
\begin{equation}
n+\Delta\varphi \le C tr_\chi \omega_{\varphi}\le Ce^{A(\varphi-\inf\varphi)}.
\end{equation}
This estimate is similar to the estimate in Yau \cite{Yau}.

In the next step, recall that Yau \cite{Yau} used $e^f\omega_0^n=\omega_\varphi^n$. This is not true in this case. However,
\begin{equation}
\int_M |\nabla e^{-N\varphi/2}|^2 = \frac{N}{4} \int_M e^{-N\varphi}\Delta\varphi \le CN e^{-A\inf\varphi} \int_M e^{-N\varphi+A\varphi}
\end{equation}
is still true for any $N>0$. In particular, we can start with $N=A+\alpha$, where $\alpha$ is Tian's $\alpha$-invariant, which means that $\int_M e^{-\alpha\varphi}\le C$ for the normalization $\sup_M\varphi=-1$. We can choose a sequence of $N$ and use the Moser iteration to show that
\begin{equation}
e^{-A\inf\varphi} = \lim_{p\to\infty}||e^{-A\varphi}||_{L^p} \le C e^{-A\inf\varphi(1-\mu)}(\int_M e^{-\alpha\varphi})^{\mu A/\alpha}
\end{equation}
for a constant $\mu\in(0,1)$. This implies a bound on $e^{-A\inf \varphi}$, which means a bound on $||\varphi||_{C^0}$. Moreover, $0\le n+\Delta\varphi \le C$ for a constant $C$.

In Datar-Pingali's paper \cite{DatarPingali}, they used a different strategy, following the approach by Sz\'{e}kelyhidi \cite{Szekelyhidi}. Their $C^0$ estimate used the Alexandrov-Bakelmann-Pucci maximum principle.
\begin{theorem}[Alexandrov-Bakelmann-Pucci maximum principle]
Let $v$ be a smooth function on $B(1)\subset \mathbb{R}^m$ such that
\begin{equation}
v(0)+\epsilon\le \inf_{\partial B(1)}v
\end{equation}
for a constant $\epsilon>0$. Let $P$ be the set
\begin{equation}
\{x\in B(1), |Dv(x)|<\frac{\epsilon}{2}, v(y)\ge v(x)+Dv(x)\cdot(y-x), \forall y\in B(1)\},
\end{equation}
then there exists a constant $C>0$ such that
\begin{equation}
C\epsilon^{m} \le \int_P \det(D^2 v).
\end{equation} 
\end{theorem}
In \cite{Szekelyhidi} and \cite{DatarPingali}, assume that $\sup\varphi=-1$, $\int_M|\varphi|\le C$, and $\inf\varphi=\varphi(0)$ in local coordinates. Define $v=\varphi+\epsilon|z|^2$ for a constant $\epsilon$ be to determined. If $x\in P$, then $D^2 v\ge 0$. So
\begin{equation}
\det(D^2 v)\le 2^{2n}\det(v_{i\bar j}).
\end{equation}
In the above, $D^2 v\in \mathbb{R}^{m\times m}$ for $m=2n$, and $v_{i\bar j}\in \mathbb{C}^{n\times n}$. The condition $D^2 v>0$ implies that $\varphi_{i\bar j}\ge -\epsilon \delta_{i\bar j}$. Think about a special case such that $\chi$, $\omega_0$, $\omega_\varphi$ are all diagonalizable, $\chi_{i\bar j}=\delta_{i\bar j}$, $\omega_{0, i\bar j}=\lambda_i\delta_{i\bar j}$, and $\omega_{\varphi, i\bar j}=\mu_i\delta_{i\bar j}$, then the cone condition for $\omega_0$ implies that
\begin{equation}
\sum_{i\not=k} \frac{1}{\lambda_i} < c-\epsilon'.
\end{equation} The condition $\varphi_{i\bar j}\ge -\epsilon\delta_{i\bar j}$ implies that $\mu_i\ge\lambda_i-\epsilon$. The J-equation is that $\sum_i \frac{1}{\mu_i}=c$, so
\begin{equation}
\frac{1}{\mu_k} = c-\sum_{i\not=k}\frac{1}{\mu_i} \ge c-\sum_{i\not=k}\frac{1}{\lambda_i-\epsilon}\ge \frac{\epsilon'}{2}
\end{equation}
if we choose $\epsilon$ small enough. This implies that $\mu_k\le C$ in this special case. In general, a slightly more complicated argument show that $|v_{i\bar j}|\le C$ on $P$. Now we apply the Alexandrov-Bakelmann-Pucci maximum principle to get
\begin{equation}
C\epsilon^{2n}\le \int_P \det(D^2 v) \le 2^{2n} \int_P \det(v_{i\bar j}) \le C Vol(P).
\end{equation}
Note note
\begin{equation}
v(0)\ge v(x)+Dv(x)\cdot(0-x) \ge v(x)-\frac{\epsilon}{2}
\end{equation}
for all $x\in P$. So
\begin{equation}
\varphi(x)\le \inf\varphi + \frac{\epsilon}{2}
\end{equation}
for all $x\in P$. This implies that
\begin{equation}
C\epsilon^{2n}\le CVol(P) \le \frac{C\int_P |\varphi|}{|\inf\varphi + \frac{\epsilon}{2}|} \le \frac{C}{|\inf\varphi + \frac{\epsilon}{2}|}
\end{equation}
So $|\inf\varphi|\le C$. In other words, we have proved that $||\varphi||_{C^0}\le C$.

Then we consider the $C^2$-estimate. For the Monge-Amp\`ere equation, we used $e^{-A\varphi}tr_{\omega_0}\omega_\varphi$. For the J-equation, we used $-A\varphi+\log (tr_{\chi}\omega_\varphi)$. In \cite{Szekelyhidi} and \cite{DatarPingali}, they used the function
\begin{equation}
G = \log \lambda_1 + \phi(|\nabla\varphi|^2) +\psi(\varphi),
\end{equation}
where $\lambda_1$ is the largest eigenvalue of $\chi^{i\bar j}\omega_{\varphi, k\bar j}$,
\begin{equation}
\phi(t)=-\frac{1}{2} \log(1-\frac{t}{2K}),
\end{equation}
\begin{equation}
K=\sup|\nabla\varphi|^2+1,
\end{equation}
and
\begin{equation}
\psi(t)=-2At+\frac{A\tau}{2}t^2
\end{equation}
for parameters $A,\tau$ determined in the proof. The calculation is more complicated, but they manage to prove that 
at the maximal point of $G$, $\lambda_1\le CK$. Then $\max G\le C+\log(K)$ by the $C^0$-bound on $\varphi$ as well as the bounds on $\psi$. This then implies that 
\begin{equation}
\log \lambda_1\le C+\log(K)
\end{equation}
on $M$, which is equivalent to
\begin{equation}
|\varphi_{i\bar j}|\le C+C\lambda_1 \le C+CK \le C(1+\sup|\nabla\varphi|^2).
\end{equation}
We will then run the contradiction argument with this bound as in Collins, Jacob, and Yau's paper \cite{CollinsJacobYau}.

If $K_i=\sup|\nabla\varphi|^2\to\infty$, then we define
\begin{equation}
\tilde\varphi_i(z)=\varphi_i(\frac{z}{K_i}).
\end{equation}
Then $|\nabla\tilde\varphi_i|\le 1$ and $|\nabla\tilde\varphi_i(0)|= \sqrt{\frac{K_i-1}{K_i}}$. The condition $i\partial\bar\partial\varphi_i\ge-\omega_0$ implies that $i\partial\bar\partial\tilde\varphi_i\ge-\frac{\omega_0}{K_i^2}$. We also have that $|\tilde\varphi|\le C$ and $|\partial\bar\partial\tilde\varphi|\le C$. The $C^0$ bound, $C^1$ bound and a bound on $|\partial\bar\partial\tilde\varphi|$ implies that a subsequence of $\tilde\varphi_i$ converges to $\varphi_\infty$ in $C^{1,\alpha}_{loc}$. Then we need to take the limit of $i\partial\bar\partial\tilde\varphi_i\ge-\frac{\omega_0}{K_i^2}$ to say that $i\partial\bar\partial\varphi_\infty\ge 0$ in some weak sense and then try to get a contradiction. There are several ways to define the weak sense.
\begin{definition}
If $\varphi$ is an $L^1$-function on $B(1)$, then we define $i\partial\bar\partial\varphi\ge 0$ in the sense of distribution if and only if for all strongly positive $(n-1,n-1)$ smooth form $\xi$, 
\begin{equation}
\int_M i\partial\bar\partial\varphi\wedge\xi = \int_M \varphi\wedge i\partial\bar\partial\xi \ge 0.
\end{equation}
\end{definition}
\begin{definition}
We choose a cut-off function $\rho\ge 0$ with $supp \rho = B(1)$, $\int_{\mathbb{C}^n} \rho=1$ and define
\begin{equation}\rho_\epsilon(z)=\frac{1}{\epsilon^{2n}}\rho(\frac{z}{\epsilon}).
\end{equation}
If $\varphi$ is an $L^1$-function on $B(1)\subset\mathbb{C}^n$, then we define the smoothing $\varphi_\epsilon$ of $\varphi$ by 
\begin{equation}
\varphi_\epsilon(x) = \int_M \varphi(y) \rho_\epsilon(x-y)dy.
\end{equation}
Then $\varphi_\epsilon\to \varphi$ in $L^1$. We say that $i\partial\bar\partial\varphi\ge 0$ in the sense of smoothing if and only if $i\partial\bar\partial \rho_\epsilon\ge 0$ for all $\epsilon>0$.
\end{definition}
\begin{proposition}
If $i\partial\bar\partial\varphi\ge 0$ in the sense of distribution if and only if $i\partial\bar\partial\varphi\ge 0$ in the sense of smoothing.
\end{proposition}
\begin{proof}
$i\partial\bar\partial\varphi\ge 0$ in the sense of smoothing, then for any strongly positive $(n-1,n-1)$ smooth form $\xi$, 
\begin{equation}
\int_M i\partial\bar\partial\varphi\wedge\xi = \int_M \varphi\wedge i\partial\bar\partial\xi =\lim_{\epsilon \to 0} \int_M \varphi_\epsilon \wedge i\partial\bar\partial\xi =\lim_{\epsilon \to 0}\int_M i\partial\bar\partial\varphi_\epsilon\wedge\xi \ge 0.
\end{equation}
Conversely, if $i\partial\bar\partial\varphi\ge 0$ in the sense of distribution, then
\begin{equation}
\int_M i\partial_x\bar\partial_x\varphi_\epsilon(x)\wedge\xi(x) = \int_M \varphi_\epsilon(x)\wedge i\partial_x\bar\partial_x\xi(x) = \int_M (\int_M \varphi(x-y)\wedge i\partial_x\bar\partial_x\xi(x))\rho_\epsilon(y)dy\ge 0.
\end{equation}
By the pointwise duality of strongly positive and positive forms, we see that $i\partial\bar\partial\varphi_\epsilon\ge 0$.
\end{proof}
Back to the contradiction argument, it is easy to see that the $C^{1,\alpha}_{loc}$ converges implies that $i\partial\bar\partial\varphi_\infty\ge 0$ in the sense of smoothing. Therefore, for any $\epsilon>0$, $i\partial\bar\partial(\varphi_\infty)_\epsilon\ge 0$ for a smooth function $(\varphi_\infty)_\epsilon$ such that $|(\varphi_\infty)_\epsilon|\le C$, $|\nabla(\varphi_\infty)_\epsilon|\le C$. The Liouville's theorem implies that $(\varphi_\infty)_\epsilon$ must be a constant for any $\epsilon>0$. So $\varphi_\infty$ is also a constant. This contradicts the condition that $|\nabla\tilde\varphi_i(0)|= \sqrt{\frac{K_i-1}{K_i}}\to 1$ and $\tilde\varphi_i\to\varphi_\infty$ in $C^{1,\alpha}_{loc}$. In other words, the contradiction argument implies a bound on $K$ and therefore, bounds on $|\varphi|$ and $|i\partial\bar\partial\varphi|$.

\subsection{Higher derivative bound}
In all of the above cases, we have proved bounds on $|\varphi|$ and $|i\partial\bar\partial\varphi|$. Standard elliptic regularity implies bounds on $|\nabla\varphi|$.

For the real Monge-Amp\`ere equation $\det u_{ij}=1$ on $\mathbb{R}^n$, Calabi proved the $C^3$-estimate \cite{CalabiC3}. For the complex Monge-Amp\`ere equation
\begin{equation}\omega_\varphi^n=e^f\omega_0^n,
\end{equation}
Yau \cite{Yau} generalized Calabi's $C^3$-estimate \cite{CalabiC3} by defining
\begin{equation}
S=\omega^{i\bar j}_{\varphi}\omega^{k\bar l}_{\varphi}\omega^{m\bar n}_{\varphi}\varphi_{i\bar lm}\varphi_{\bar jk\bar n}.
\end{equation}
Yau proved that
\begin{equation}
\Delta_{\varphi}(S+A\Delta\varphi)\ge CS-C
\end{equation}
for a constant $A$. At the maximal point of $S+A\Delta\varphi$, $S\le C$. So $S+A\Delta\varphi \le C$ at this point and also other points because of the bound on $\Delta\varphi$. This implies a bound on $S$. So $[\varphi_{i\bar j}]_{C^\alpha}\le C$. Higher derivative bounds then follow Scauder's estimate.

In general, we use Evans-Krylov's estimate to get $[\varphi_{i\bar j}]_{C^\alpha}\le C$ and then apply Scauder's estimate to get higher derivative bounds. Evans-Krylov's estimate for the equation $\Phi(\varphi_{i\bar j})=0$ requires $|\varphi|\le C$, $|\nabla \varphi|\le C$, $|\varphi_{i\bar j}|\le C$, the uniform ellipticity
\begin{equation}
C^{-1} \le \frac{\partial\Phi}{\partial A_{i\bar j}}\le C,
\end{equation}
and concavity
\begin{equation}
\frac{\partial^2\Phi}{\partial A_{i\bar j}\partial A_{k\bar l}}\le 0,
\end{equation} which is equivalent to
\begin{equation}
\Phi(tB+(1-t)A)\ge t\Phi(B)+(1-t)\Phi(A).
\end{equation}
For the Monge-Amp\`ere equation, 
\begin{equation}
\Phi=\log (\lambda_1...\lambda_n),
\end{equation}
so
\begin{equation}
\frac{\partial\Phi}{\partial \lambda_i}=\frac{1}{\lambda_i},
\end{equation}
and
\begin{equation}
\frac{\partial^2\Phi}{\partial \lambda_i\lambda_j}=-\frac{1}{\lambda_i^2}\delta_{ij},
\end{equation}
The derivatives on matrices involve a chain rule, but the required properties are true. For the J-equation,
\begin{equation}
\Phi=-\sum_{i}\frac{1}{\lambda_i},
\end{equation}
so
\begin{equation}
\frac{\partial\Phi}{\partial \lambda_i}=\frac{1}{\lambda_i^2},
\end{equation}
and
\begin{equation}
\frac{\partial^2\Phi}{\partial \lambda_i\lambda_j}=-\frac{2}{\lambda_i^3}\delta_{ij},
\end{equation}
The condition can be verified similarly, even with the twisted term $\frac{f}{\prod_i\lambda_i}$. The twisted equation in \cite{DatarPingali} is more complicated but is similar.

\section{Existence theorems and subvarieties}
Now we start to understand Demailly-Paun's theorem \cite{DemaillyPaun} to see how subvarieties come in. We start with several definitions.
\begin{definition}
If $\omega_t>0$ are smooth K\"ahler forms and $[\omega_t]\to[\omega_0]$, then we say that $[\omega_0]$ is a nef class.
\end{definition}
\begin{definition}
We $[\omega_0]$ big if and only if $\int_M\omega_0^n>0$.
\end{definition}
It suffices to prove that any nef and big class $[\omega_0]$ is a K\"ahler class if $\int_{V^p}\omega_0^p\ge 0$ for all analytic subvariety $V^p\subset M^n$. If we just take the weak limit of $\omega_t$, then $[\omega_0]$ contains a positive current (in the sense of distribution). A key proposition is the following:
\begin{proposition}[\cite{DemaillyPaun}]
Any nef and big class $[\omega_0]$ contains a K\"ahler current $\omega_\varphi=\omega_0+i\partial\bar\partial\varphi$ for an $L^1$ potential $\varphi$. If $\chi$ is a background K\"ahler form, then we say $\omega_\varphi$ a K\"ahler current if there exists $\epsilon>0$ such that $\omega_\varphi-\epsilon\chi$ is still a positive current.
\end{proposition}
\begin{proof}
The key point is to get the extra $\epsilon$ using the nef and big condition. They used Yau's solution to the Calabi conjecture in this process. Firstly assume that $V^p\subset M^n$ is a smooth subvariety locally given by $z_1=...=z_{n-p}=0$.
Then we define
\begin{equation}
\chi_{\epsilon,\delta}=\chi+\epsilon \cdot i\partial\bar\partial \log(\sum_{i}|z_i^2|+\delta)\ge\frac{1}{2}\chi
\end{equation}
locally and then glue them together. If we fix a sufficiently small $\epsilon$ and let $\delta\to 0$, then because $(i\partial\bar\partial \log(\sum_{i=1}^{n-p}|z_i^2|))^{n-p}$ is a constant multiple of the Dirac delta function at 0 on $\mathbb{C}^{n-p}$, we can get
\begin{equation}
\lim_{\delta\to0}\chi^{n-p}_{\epsilon,\delta}|_V=\epsilon'[V]
\end{equation}
for $\epsilon'>0$. In other words, if $T_\mu(V)$ is the tubular neighborhood of $V$ with radius $\mu$, then
\begin{equation}
\lim_{\mu\to0}\lim_{\delta\to0}\int_{T_\mu(V)}\chi^{n-p}_{\epsilon,\delta}\wedge\xi=\epsilon'\int_V\xi
\end{equation}
for all $(p,p)$-form $\xi$. Then we use the nef condition $\omega_t>0$, the big condition 
\begin{equation}
\lim_{t\to0}\int_M\omega_t^n\to\int_M\omega_0^n>0,
\end{equation}
and Yau's solution \cite{Yau} to the Calabi conjecture to get a smooth K\"ahler form $\omega_{t,\epsilon,\delta}=\omega_t+i\partial\bar\partial\varphi_{t,\epsilon,\delta}$ such that
\begin{equation}
\omega_{t,\epsilon,\delta}^n=c_t \chi_{\epsilon,\delta}^n
\end{equation}
for a constant $c_t$ such that
\begin{equation}
\int_M\omega_t^n=c_t \int_M\chi^n.
\end{equation}
Then we should also expect that
\begin{equation}
\lim_{\delta\to0}\omega^{n-p}_{t,\epsilon,\delta}|_V=\epsilon''[V]
\end{equation}
for another constant $\epsilon''>0$. To see this, let $\lambda_1>...>\lambda_n$ be eigenvalues of $\chi_{\epsilon,\delta}^{i\bar k}\omega_{t,\epsilon,\delta,j\bar k}$ and $E$ be the set $\lambda_1...\lambda_p \ge \frac{1}{\delta'}$. Then
\begin{equation}
C\ge \int_M\omega_{t,\epsilon,\delta}^p\wedge\chi_{\epsilon,\delta}^{n-p}\ge \int_E \lambda_1...\lambda_p\chi_{\epsilon,\delta}^{n} \ge \frac{1}{\delta'}\int_E \chi_{\epsilon,\delta}^{n}.
\end{equation}
So 
\begin{equation}
\int_E \chi_{\epsilon,\delta}^{n-p}\wedge\chi^p \le 2^p \int_E \chi_{\epsilon,\delta}^{n}\le C \delta'\to 0
\end{equation}
as $\delta'\to 0$. If we choose $\delta'$ small enough and use the condition that \begin{equation}
\lim_{\delta\to0}\chi^{n-p}_{\epsilon,\delta}|_V=\epsilon'[V],
\end{equation}
then the limit
\begin{equation}
\lim_{\mu\to0}\lim_{\delta\to0}\int_{T_\mu(V)\cap E^c}\chi^{n-p}_{\epsilon,\delta}\wedge \chi^p \ge C\epsilon'.
\end{equation}
On the set $E^c$, we have $\lambda_1...\lambda_p < \frac{1}{\delta'}$, which implies that $\lambda_{p+1}...\lambda_n>C\delta'$. So
\begin{equation}
\lim_{\mu\to0}\lim_{\delta\to0}\int_{T_\mu(V)\cap E^c}\omega^{n-p}_{t,\epsilon,\delta}\wedge \chi^p \ge C\delta' \lim_{\mu\to0}\lim_{\delta\to0}\int_{T_\mu(V)\cap E^c}\chi^{n-p}_{\epsilon,\delta}\wedge \chi^p \ge C\epsilon'^2.
\end{equation}
This implies the required estimate 
\begin{equation}
\lim_{\delta\to0}\omega^{n-p}_{t,\epsilon,\delta}|_V=\epsilon''[V].
\end{equation}
by the Skoda-El Mir extension theorem and support theorems. See \cite{Demailly} for more details.

Now we apply this argument to $M\times M$ and $V$ be the diagonal $D$. Then \begin{equation}
\lim_{\delta\to0}\omega^{n}_{t,\epsilon,\delta}|_D=\epsilon''[D].
\end{equation}
and $\omega_{t,\epsilon,\delta}>0$. If we consider the current
\begin{equation}
\lim_{t\to 0}\lim_{\delta\to0}(\pi_1)_* (\omega_{t,\epsilon,\delta}^n\wedge\pi_2^*\chi),
\end{equation}
then the term $\epsilon''[D]$ provides the extra $\epsilon'\chi$ and the remaining part is still a positive current. This provides the required K\"ahler current.
\end{proof}

Once we have the K\"ahler current, then we can use the Bergmann approximation. Roughly speaking, we define
\begin{equation}
||f||_{L^2_{k\varphi}}=\int |f|^2 e^{-2k\varphi}
\end{equation}
for a large enough integer $k$. The condition that $\omega_\varphi$ is a K\"ahler current in some sense implies a curvature condition and therefore, a vanishing theorem for the obstruction to extending local holomorphic functions. This implies that
\begin{equation}
\frac{1}{2k} i\partial\bar\partial \log(\sum_{i} |f_i|^2)
\end{equation}
is an approximation to $i\partial\bar\partial \varphi$, where $f_i$ is an orthonormal basis of holomorphic functions using the $L^2_{k\varphi}$ norm. The actual statement involves a gluing argument and the additional $\epsilon\chi$ cancels the error terms. Remark that the Bergman approximation is still not a smooth form because $V=\{f_i=0\}$ may not be empty. However, $V$ is a subvariety with $dim V<dim M$. Recall that the condition is that $\int_{Y^p} \omega_0^p>0$ for all subvariety $Y^p\subset M^n$. This condition is also satisfied by $[\omega_0]|_V$. Thus, if $V$ is smooth and if we use the induction on $dim M$, then there exists a smooth K\"ahler form 
\begin{equation}
\omega_V=\omega_0+i\partial\bar\partial \varphi_V
\end{equation} near $V$. We can take 
\begin{equation}\omega_0+i\partial\bar\partial \tilde\max\{\varphi_V,\frac{1}{2k} i\partial\bar\partial \log(\sum_{i} |f_i|^2)\}\end{equation} to get the required smooth K\"ahler form on $M$, where 
\begin{equation}
\tilde\max\{f,g\}=\int_{-\infty}^{\infty} \max\{f-t,g\} \rho_{\epsilon}(t) dt
\end{equation} is the regularized maximum function for a smooth function $\rho_{\epsilon}$ supported on $[-\epsilon,\epsilon]$. Remark that $\tilde\max\{f,g\}=f$ if $f>g+\epsilon$, $\tilde\max\{f,g\}=g$ if $f<g-\epsilon$, and $\tilde\max\{f,g\}$ is smooth if $f$ and $g$ are smooth. Finally, if $V$ is not smooth, we use Hironaka's resolution of singularity to get smooth submanifolds.

We have sketched the results in Demailly-Paun's paper \cite{DemaillyPaun}. Then we consider the J-equation as in \cite{Chen}. The first problem is that we can not talk about ``positive" current in the sense of distribution, because
\begin{equation}
c\omega_\varphi^{n-1}-(n-1)\omega_\varphi^{n-2}\wedge\chi
\end{equation}
is not well-defined for unbounded $\varphi$. To solve this problem, we define the cone condition
\begin{equation}
c\omega_\varphi^{n-1}-(n-1)\omega_\varphi^{n-2}\wedge\chi>0
\end{equation}
in the sense of smoothing. With this definition, we can use a similar argument to concentrate the mass on the diagonal of $M\times M$ and then push down the limit current. It can be proved that if $\omega_t$ are smooth forms that satisfy the cone condition for all $t>0$ and $[\omega_t]\to[\omega_0]$, then there exists $\epsilon>0$ such that $\omega_0-\epsilon\chi$ contains a current which satisfies the cone condition in the sense of smoothing.

For the generalized Monge-Amp\`ere equation, however, this argument fails. In fact,
\begin{equation}
\log(\lambda_1...\lambda_{2n}) = \log(\lambda_1...\lambda_{n})+\log(\lambda_{n+1}...\lambda_{2n}),
\end{equation}
and
\begin{equation}
\sum_{i=1}^{2n}\frac{1}{\lambda_i} = \sum_{i=1}^{n}\frac{1}{\lambda_i} + \sum_{i=n+1}^{2n}\frac{1}{\lambda_i},
\end{equation}
but
\begin{equation}
\sigma_k(\{\frac{1}{\lambda_i}\}|_{i=1}^{2n}) \not= \sigma_k(\{\frac{1}{\lambda_i}\}|_{i=1}^{n})+\sigma_k(\{\frac{1}{\lambda_i}\}|_{i=n+1}^{2n}).
\end{equation}
So the two components of $M\times M$ do not behave well for the generalized Monge-Amp\`ere equation. To solve this issue, Datar and Pingali \cite{DatarPingali} use the projective condition. They concentrate the mass on a very ample divisor to get the additional $\epsilon\chi$.

The next issue is that the Bergmann approximation only works for positive/K\"ahler currents. To solve this issue, the author \cite{Chen} used local smoothing to approximate currents by smooth functions. For example, suppose that $M=T^{2n}=\mathbb{C}^n/\mathbb{Z}^{2n}$, and 
\begin{equation}
\chi=i d z^i\wedge d\bar z^i.
\end{equation}
Then $M=\cup_j B(x_j,r)$. Suppose that $f_j$ are local smooth functions defined on $B(x_j,2r)$ such that
\begin{equation}
|f_j-f_{j'}|<\frac{r^2}{100}
\end{equation}
on $B(x_j,2r)\cap B(x_{j'},2r)$. We define
\begin{equation}
f=\tilde\max\{f_j-|z-x_j|^2\}.
\end{equation}
Then locally, $i\partial\bar\partial f$ is close to
\begin{equation}
i\partial\bar\partial f_j- i\partial\bar\partial f_j|z-x_j|^2 = i\partial\bar\partial f_j-\chi.
\end{equation}
Even though we change the class by $\chi$, the key point is that $f$ must be smooth. In fact, the only issue for the non-smoothness is when the regularized maximum involves $f_j- |z-x_j|^2$ but $z\in \partial B(x_j,2r)$. This can not happen if we choose the constant $\epsilon$ in the definition of regularized maximum to be smaller than $r^2$ because there exists $j'$ such that $z\in B(x_{j'},r)$, and therefore
\begin{equation}
f_j-|z-x_j|^2 = f_j-4r^2 \le f_{j'} +\frac{r^2}{100}-4r^2< f_{j'}-|z-x_{j'}|^2 -2r^2.
\end{equation}
This kind of argument implies that we can glue the local potentials on $M$ by changing the cohomology class a little bit if the local potentials are close to each other. This can be viewed as a generalization of B{\l}ocki and Ko{\l}odziej's argument \cite{BlockiKolodziej}.

The next question is: when are the smoothings of $\varphi$ in different local coordinates close to each other? The answer is that this is true if the Lelong number $\nu(x)$ is small. A Theorem by Siu \cite{Siu} shows that for any $\epsilon>0$, the set $V_\epsilon=\{\nu(x)>\epsilon\}$ is an analytic subvariety. Then we can use the condition on the integral on subvarieties, the induction on $dim M$ and Hironaka's resolution of singularity to get a smooth potential function near $V_\epsilon$ satisfying the cone condition. Then we glue this function with the smoothing of $\varphi$ in local coordinates to get a smooth potential function satisfying the cone condition. This process requires the additional $\epsilon\chi$ obtained in the previous step. Remark that there is a technical issue in this step which makes the results in \cite{Chen} slightly weaker. This technical issue has been solved by Song \cite{Song}, Datar, and Pingali \cite{DatarPingali}.

Finally, as in \cite{BlockiKolodziej}, let us understand the Lelong number in more detail. If $\varphi$ is a plurisubharmonic on $\mathbb{C}^n$ (i.e. $i\partial\bar\partial\varphi\ge 0$ in the sense of distribution or equivalently, in the sense of smoothing), then we define
\begin{equation}
\hat\varphi_\delta(x)=\max_{B_\delta(x)}\varphi.
\end{equation}
This is a convex function in $\log\delta$. So the function
\begin{equation}
\nu(x,\delta)=\frac{\hat\varphi_\delta(x)-\hat\varphi_r(x)}{\log\delta-\log r}
\end{equation}
is non-decreasing in $\log\delta$. The limit $\lim_{\delta\to 0}\nu(x,\delta)$ is called the Lelong number of $\varphi$ at $x$. Using the convexity,
\begin{equation}
\frac{\hat\varphi_\delta(x)-\hat\varphi_{\delta/2}(x)}{\log\delta-\log (\delta/2)} < \nu(x,\delta).
\end{equation}
So
\begin{equation}
|\hat\varphi_\delta(x)-\hat\varphi_{\delta/2}(x)|< \log 2 \cdot \nu(x,\delta).
\end{equation}
If we have two coordinates $U_i$, $U_j$, then
\begin{equation}
B^i(x,C^{-1}\delta)<B^j(x,\delta)<B^i(x,C\delta)
\end{equation}
So $\hat\varphi_\delta(x)$ are close to each other if the Lelong number is small. Next, we need to compare the smoothing $\varphi_\delta(x)$ with $\hat\varphi_\delta(x)$. We define
\begin{equation}
\tilde\varphi_\delta(x)=\frac{\int_{\partial B_\delta(x)}\varphi}{Vol(\partial B_\delta(x))}.
\end{equation}
Then 
\begin{equation}
\varphi_\delta(x)=\frac{1}{C}\int_0^\delta r^{2n-1}\tilde\varphi_r(x)\rho_\delta(r)dr.
\end{equation}
We use the Poisson kernel to get
\begin{equation}
0\le \hat\varphi_\delta(x)- \tilde\varphi_\delta(x) \le \frac{3^{2n-1}}{2^{2n-2}}(\hat\varphi_\delta(x)-\hat\varphi_{\delta/2}(x)) \le C\nu(x,\delta).
\end{equation}
Then
\begin{equation}
\begin{split}
0 &\le \hat\varphi_\delta(x)- \varphi_\delta(x) \\
&= \frac{1}{C}\int_0^\delta r^{2n-1}(\hat\varphi_\delta(x)-\tilde\varphi_r(x))\rho_\delta(r)dr\\
&\le C\nu(x,\delta)+\frac{1}{C}\int_0^\delta r^{2n-1}(\hat\varphi_\delta(x)-\hat\varphi_r(x))\rho_\delta(r)dr \\
&\le C\nu(x,\delta).
\end{split}
\end{equation}

\bibliographystyle{amsalpha}

\bibliography{Hodge}

\providecommand{\bysame}{\leavevmode\hbox to3em{\hrulefill}\thinspace}
\providecommand{\MR}{\relax\ifhmode\unskip\space\fi MR }
% \MRhref is called by the amsart/book/proc definition of \MR.
\providecommand{\MRhref}[2]{%
  \href{http://www.ams.org/mathscinet-getitem?mr=#1}{#2}
}
\providecommand{\href}[2]{#2}
\begin{thebibliography}{{Son}20}

\bibitem[AH62]{AtiyahHirzebruch}
M.~F. Atiyah and F.~Hirzebruch, \emph{Analytic cycles on complex manifolds},
  Topology \textbf{1} (1962), 25--45. \MR{145560}

\bibitem[BH17]{BabaeeHuh}
Farhad Babaee and June Huh, \emph{A tropical approach to a generalized {H}odge
  conjecture for positive currents}, Duke Math. J. \textbf{166} (2017), no.~14,
  2749--2813. \MR{3707289}

\bibitem[BK07]{BlockiKolodziej}
Zbigniew B{\l}ocki and S{\l}awomir Ko{\l}odziej, \emph{On regularization of
  plurisubharmonic functions on manifolds}, Proc. Amer. Math. Soc. \textbf{135}
  (2007), no.~7, 2089--2093. \MR{2299485}

\bibitem[Cal70]{CalabiC3}
Eugenio Calabi, \emph{Examples of {B}ernstein problems for some nonlinear
  equations}, Global {A}nalysis ({P}roc. {S}ympos. {P}ure {M}ath., {V}ol. {XV},
  {B}erkeley, {C}alif., 1968), Amer. Math. Soc., Providence, R.I., 1970,
  pp.~223--230. \MR{0264210}

\bibitem[Che21]{Chen}
Gao Chen, \emph{The {J}-equation and the supercritical deformed
  {H}ermitian-{Y}ang-{M}ills equation}, Invent. Math. \textbf{225} (2021),
  no.~2, 529--602. \MR{4285141}

\bibitem[Cho49]{Chow}
Wei-Liang Chow, \emph{On compact complex analytic varieties}, Amer. J. Math.
  \textbf{71} (1949), 893--914. \MR{33093}

\bibitem[CJY20]{CollinsJacobYau}
Tristan~C. Collins, Adam Jacob, and Shing-Tung Yau, \emph{{$(1,1)$} forms with
  specified {L}agrangian phase: a priori estimates and algebraic obstructions},
  Camb. J. Math. \textbf{8} (2020), no.~2, 407--452. \MR{4091029}

\bibitem[Dem97]{Demailly}
Jean-Pierre Demailly, \emph{{Complex analytic and differential geometry}}.

\bibitem[DP04]{DemaillyPaun}
Jean-Pierre Demailly and Mihai Paun, \emph{Numerical characterization of the
  {K}\"{a}hler cone of a compact {K}\"{a}hler manifold}, Ann. of Math. (2)
  \textbf{159} (2004), no.~3, 1247--1274. \MR{2113021}

\bibitem[DP21]{DatarPingali}
Ved~V. Datar and Vamsi~Pritham Pingali, \emph{A numerical criterion for
  generalised {M}onge-{A}mp\`ere equations on projective manifolds}, Geom.
  Funct. Anal. \textbf{31} (2021), no.~4, 767--814. \MR{4317503}

\bibitem[Hod52]{Hodge}
W.~V.~D. Hodge, \emph{The topological invariants of algebraic varieties},
  Proceedings of the {I}nternational {C}ongress of {M}athematicians,
  {C}ambridge, {M}ass., 1950, vol. 1, Amer. Math. Soc., Providence, R.I., 1952,
  pp.~182--192. \MR{0046075}

\bibitem[Kle66]{Kleiman}
Steven~L. Kleiman, \emph{Toward a numerical theory of ampleness}, Ann. of Math.
  (2) \textbf{84} (1966), 293--344. \MR{206009}

\bibitem[Lel57]{Lelong}
Pierre Lelong, \emph{Int\'{e}gration sur un ensemble analytique complexe},
  Bull. Soc. Math. France \textbf{85} (1957), 239--262. \MR{95967}

\bibitem[Siu74]{Siu}
Yum~Tong Siu, \emph{Analyticity of sets associated to {L}elong numbers and the
  extension of closed positive currents}, Invent. Math. \textbf{27} (1974),
  53--156. \MR{352516}

\bibitem[{Son}20]{Song}
Jian {Song}, \emph{{Nakai-Moishezon criterions for complex Hessian equations}},
  arXiv e-prints (2020), arXiv:2012.07956.

\bibitem[SW08]{SongWeinkove}
Jian Song and Ben Weinkove, \emph{On the convergence and singularities of the
  {$J$}-flow with applications to the {M}abuchi energy}, Comm. Pure Appl. Math.
  \textbf{61} (2008), no.~2, 210--229. \MR{2368374}

\bibitem[Sz{\'{e}}18]{Szekelyhidi}
G\'{a}bor Sz{\'{e}}kelyhidi, \emph{Fully non-linear elliptic equations on
  compact {H}ermitian manifolds}, J. Differential Geom. \textbf{109} (2018),
  no.~2, 337--378. \MR{3807322}

\bibitem[Voi02]{Voisin}
Claire Voisin, \emph{A counterexample to the {H}odge conjecture extended to
  {K}\"{a}hler varieties}, Int. Math. Res. Not. (2002), no.~20, 1057--1075.
  \MR{1902630}

\bibitem[Yau78]{Yau}
Shing~Tung Yau, \emph{On the {R}icci curvature of a compact {K}\"{a}hler
  manifold and the complex {M}onge-{A}mp\`ere equation. {I}}, Comm. Pure Appl.
  Math. \textbf{31} (1978), no.~3, 339--411. \MR{480350}

\end{thebibliography}

 \end{document}